\documentclass{amsart}
\title{Mutually embeddable models of ZFC}
\author{Monroe Eskew}
\author{Sy-David Friedman}
\author{Yair Hayut}
\author{Farmer Schlutzenberg}
\thanks{The first two authors would like to acknowledge the support of the Austrian Science Fund (FWF) through Project P28420.}
\thanks{The third author's research was partially supported by the FWF Lise Meitner Grant (2650-N35).}
\thanks{The fourth author was supported by the Deutsche Forschungsgemeinschaft (DFG, German Research Foundation) under Germany's Excellence Strategy EXC 2044--390685587, Mathematics M\"unster: Dynamics–Geometry–Structure.}
	
\address{Monroe Eskew and Sy-David Friedman: Universität Wien \\
Institut für Mathematik \\
Kurt Gödel Research Center \\
Kolingasse 14-16 \\
1090 Wien \\
AUSTRIA}
\address{Yair Hayut: 
Einstein Institute of Mathematics\\
Edmond J. Safra Campus\\
The Hebrew University of Jerusalem\\
Givat Ram. Jerusalem, 9190401, ISRAEL}
\address{Farmer Schlutzenberg: WWU M\"unster\\
	Institut f\"ur Mathematische Logik und Grundlagenforschung\\
	Einsteinstra{\ss}e 62\\
	48149 M\"unster\\
	GERMANY}

\date{}

\usepackage{tikz-cd}
\usetikzlibrary{cd}
\usepackage{color} 
\usepackage{amssymb} 
\usepackage{amsmath} 
\usepackage[utf8]{inputenc} 
\usepackage{enumerate}
\usepackage{amsthm}
\usepackage{cite}
\usepackage{comment}

\newtheorem{theorem}{Theorem}
\newtheorem{lemma}[theorem]{Lemma}
\newtheorem{proposition}[theorem]{Proposition}

\newtheorem{question}{Question}
\newtheorem{claim}[theorem]{Claim}
\newtheorem*{definition}{Definition}
\newtheorem{remark}[theorem]{Remark}

\DeclareMathOperator{\dom}{dom}
\DeclareMathOperator{\ran}{ran}

\DeclareMathOperator{\rank}{rank}
\DeclareMathOperator{\cf}{cf}

\DeclareMathOperator{\add}{Add}
\DeclareMathOperator{\col}{Col}
\DeclareMathOperator{\ord}{Ord}

\DeclareMathOperator{\crit}{crit}

\DeclareMathOperator{\len}{len}
\DeclareMathOperator{\ult}{Ult}

\newcommand{\range}{\mathrm{range}}
\newcommand{\rg}{\mathrm{range}}
\newcommand{\lh}{\mathrm{lh}}
\newcommand{\om}{\omega}
\newcommand{\p}{\mathcal{P}}
\newcommand{\la}{\langle}
\newcommand{\ra}{\rangle}
\newcommand{\OR}{\mathrm{Ord}}
\newcommand{\ZF}{\mathrm{ZF}}
\newcommand{\ZFC}{\mathrm{ZFC}}
\newcommand{\inter}{\cap}
\newcommand{\sub}{\subseteq}
\newcommand{\HOD}{\mathrm{HOD}}
\newcommand{\sats}{\models}
\newcommand{\rest}{\upharpoonright}
\newcommand{\Ord}{\OR}

\begin{document}

\begin{abstract}
We investigate systems of transitive models of $\ZFC$ which
are elementarily embeddable into each other and the influence of definability properties on such systems.
\end{abstract}
\maketitle


One of Ken Kunen’s best-known and most striking results is that there is no
elementary embedding of the universe of sets into itself other than the identity.
Kunen’s result is best understood in a theory that includes
proper classes as genuine objects, such as von Neumann-Gödel-Bernays set theory
(NBG), which we take in this article as our background theory.
The argument proceeds by assuming there is a nontrivial $j : V \to V$ and using Replacement and Comprehension for formulas involving the parameter $j$, arriving at a contradiction.
It does not exclude the possibility of the existence of an inner model $M$ of Zermelo-Fraenkel set theory (ZFC) and a nontrivial elementary $j : M \to M$, where $M$ is not the class of sets of an NBG-universe that includes the class $j$.  Indeed, Kunen’s early work on $0^\sharp$ showed an equivalence between the nontrivial self-embeddability of $L$ and the existence of a certain real number external to $L$.
In this paper we take up a related theme, that of \emph{mutual} embeddability between
models of set theory.

In \cite{2019arXiv190506062E}, the first two authors asked whether there can be 
two transitive models $M \neq N$ of $\ZFC$ (necessarily with the same ordinals) 
such that each is elementarily embeddable into the other; that is, such that 
there are elementary embeddings $j\colon M\to N$ and $k\colon N\to M$.  We say
such models are \emph{mutually embeddable}. We answer the 
question affirmatively with the following result:

\begin{theorem}\label{thm:0^sharp}
The following are equivalent:
\begin{itemize}
\item $0^\sharp$ exists.
\item There are transitive models $M \neq N$ of $\ZFC$ containing all ordinals which are mutually embeddable. 
\end{itemize}
\end{theorem}

There are natural variants of the original 
question, given by making more demands on either the models $M,N$ or the 
embeddings $j,k$ witnessing
mutual embeddability. For example, what can hold regarding
whether $M\sub N$ or $N\sub M$? To what extent can $j,k$
be amenable to $M,N$ (that is, how much of $j,k$ can belong to $M,N$)? Can we 
arrange that $j\rest\Ord=k\rest\Ord$?
Can we arrange that 
$M,N\sats$``$V=\HOD$''?
Of course by Kunen's famous inconsistency result, $k\circ j$
cannot be an amenable class of $M$. 
A second restriction is the following:

\begin{proposition}\label{prop:3_models} Let $M_1,M_2,N$
be models of $\ZF$, $M_1 \cap \Ord = M_2 \cap \Ord$. Let us assume that there are elementary embeddings $j_1 \colon M_1 \to N$ and $j_2 \colon M_2 \to N$, such that $j_1 \restriction \Ord = j_2 \restriction \Ord$. Then $\HOD^{M_1} = \HOD^{M_2}$ and $j_1 \restriction \HOD^{M_1} = j_2 \restriction \HOD^{M_2}$.
\end{proposition}
\begin{proof}
Let $X_1 \subseteq \Ord$ with $X_1 \in \HOD^{M_1}$. Let $\varphi(x, y)$ be a formula and $\eta \in \Ord$ such that $X_1 = \{\alpha \in \Ord \mid M_1 \models \varphi(\alpha, \eta)\}$. Let $X_2 = \{\alpha \in \Ord \mid M_2 \models \varphi(\alpha, \eta)\}$. 

By elementarity for $\iota \in \{1, 2\}$, 
\[j_\iota(X_\iota) = \{\alpha \in \Ord^N \mid N \models \varphi(\alpha, j_\iota(\eta))\}.\]
Since $j_1(\eta) = j_2(\eta)$, we conclude that $j_1(X_1) = j_2(X_2)$. In particular, for every $\beta \in \Ord$, $\beta \in X_2$ iff $j_2(\beta) \in j_2(X_2)$ iff $j_1(\beta) \in j_1(X_1)$ iff $\beta \in X_1$.
\end{proof}
 


We conclude that given a diagram of the form:
\begin{center}
\begin{tikzcd}
& {N_1}\arrow[rd, "k_1"] & \\
M \arrow[ru, "j_1"]\arrow[rd, "j_2"] & & M \\
& N_2\arrow[ru, "k_2"] &
\end{tikzcd}
\end{center}
where $k_1 \restriction \Ord = k_2 \restriction \Ord$, we have that $\HOD^{N_1} = \HOD^{N_2}$. So if (for example) $M \models V = \HOD$, then $N_1 = N_2$, and therefore we cannot have three distinct models of ``$V = \HOD$'' that are pairwise mutually embeddable via embeddings that all agree on $\Ord$.
This does not rule out the possibility of a cycle of the form $M \to^{j} N \to^{k} M$, where $j \restriction \Ord = k \restriction \Ord$, $M = \HOD^{M} \neq N = \HOD^N$.  We show this is consistent with Theorem \ref{cor:hod-2-cycle}.

In \S\ref{sec:0^sharp}, we prove Theorem \ref{thm:0^sharp} and related results.  These constructions all involve a collection of generic extensions of $L$, each of which is a subclass of $V$, that are mutually embeddable via liftings of embeddings generated by Silver indiscernibles.  In \S\ref{sec:measurables}, we give a different construction using many measurable cardinals.  This method has the advantage that all of the models in the system can satisfy any given theory consistent with ZFC + many measurables.  In \S\ref{sec:2-cycles}, we address some remaining questions using a different type of construction we call strong 2-cycles.  With these techniques, we are able to build systems of more than two mutually embeddable models such that the embeddings all agree on the ordinals.  These constructions require significantly greater consistency strength, which we attempt to calibrate.  We conclude with some open questions.

The central idea of \S\ref{sec:2-cycles} was first observed by the fourth author using rank-into-rank embeddings, and then the consistency strength was reduced below a strong cardinal.  Afterwards, the construction of \S\ref{sec:0^sharp} using the optimal consistency strength was found, followed by the arguments of \S\ref{sec:measurables}.


\section{Constructions from $0^\sharp$}\label{sec:0^sharp}
In this section we prove the following strengthening of Theorem \ref{thm:0^sharp}.
\begin{theorem}
\label{main0sharp}
The following are equivalent.
\begin{enumerate}
\item $0^\sharp$ exists.
\item There are transitive models $M,N$ of $\ZFC$ containing all ordinals  
or $\OR\inter M=\OR\inter N$ is a cardinal and there is a non-identity
$j\colon M\to N$.
\item\label{item:mutual_pair} (a) There are two proper class transitive models of $\ZFC$, $M \nsubseteq N \nsubseteq M$, and elementary embeddings $j : M \to N$, $k : N \to M$ that agree on the ordinals. (b) We may take $M=L[G]$ and $N=L[H]$ to be generic extensions of $L$, having a common generic extension $L[G,H]$, such that $j,k$ extend
to an elementary $j^+:L[G,H]\to L[G,H]$ with  $j^+(G,H)=(H,G)$.
\item\label{item:mutual_family} For every cardinal $\kappa$, there is a collection of $\kappa$-many proper 
class transitive models of ZFC 
which are pairwise mutually embeddable
and $\subseteq$-incomparable.
\end{enumerate}
We can moreover arrange in part  (\ref{item:mutual_pair})
 that $M,N,j,k$ are definable over $V$ without 
parameters. In part (\ref{item:mutual_family}), for $\kappa = 2^\omega$, the family 
and the system of witnessing embeddings can be taken definable without 
parameters.
\end{theorem}
\begin{proof}
If $j:M\to N$ where $M,N\sats\mathrm{ZFC}$
is non-identity and $\lambda=\OR^M=\OR^N$ is either a cardinal or 
$=\OR$,
then of course $j\rest L_\lambda:L_\lambda\to L_\lambda$, so by Kunen's work,
$0^\sharp$ exists.

So assume $0^\sharp$ exists. 
Let $\la \kappa_i : i \in \ord \ra$ enumerate the closed class of Silver indiscernibles. 
Each $\kappa_i$ is weakly compact in $L$, but $|\kappa_i| = |i| + 
\aleph_0$ in $V$.  The key well-known fact that enables our construction is:

\begin{lemma}
For each $x \in L$, there is a 
definable Skolem function $\tau(v_0,\dots,v_{n})$ and Silver indiscernibles $\kappa_{i_0}<\dots<\kappa_{i_n}$ such that $x = 
\tau^L(\kappa_{i_0},\dots,\kappa_{i_n})$. Moreover, if $\rank(x)<\kappa_{i_k}$
and $\kappa_{i_k}\leq\kappa_{i_k'}<\ldots<\kappa_{i_n'}$ are also Silver indiscernibles
then $x=\tau^L(\kappa_{i_0},\ldots,\kappa_{i_{k-1}},\kappa_{i_k'},\ldots,\kappa_{i_n'})$.
\end{lemma}

\newcommand{\PP}{\mathbb{P}}
We give a construction that works for a general class of forcings:
\begin{definition}
A class function $\mathbb{P}:[\OR]^2\to V$ is \emph{good} iff
\begin{enumerate}
	\item $\PP\sub L$ is $L$-definable without parameters,
	\item\label{item:distrib} for all $L$-indiscernibles $\kappa$,
	and either $\mu=0$ or $\mu<\kappa$ also $L$-indiscernible,
	$L\models$``$\PP(\mu,\kappa)\subseteq L_\kappa$ is a $\mu^+$-distributive forcing with largest condition $1$'',
	\item\label{item:factoring} for all $L$-indiscernibles $\mu<\kappa$,
	$L\models$``$\PP(0,\mu)\times\PP(\mu,\kappa)\cong\PP(0,\kappa)$,
	as witnessed by $\pi:\PP(0,\mu)\times\PP(\mu,\kappa)\to\PP(0,\kappa)$
	with $\pi(p,1)=p$ for all $p\in\PP(0,\mu)$''.
	\end{enumerate}
	\end{definition}
Note that condition (\ref{item:factoring}) makes sense
because $\PP(0,\mu)=\PP(0,\kappa)\cap L_\mu$
whenever $\mu<\kappa$ are $L$-indiscernibles.

For example, work in $L$. Let $\mathcal C$ be a class of ordinals defined without parameters that (in $V$) includes all Silver indiscernibles.  For an ordinal $\alpha$, let $\alpha^{+\mathcal C}$ denote the least ordinal in $\mathcal C$ above $\alpha$. Assume that for each $\kappa \in \mathcal C$, $\mathbb Q(\kappa)$ is a nontrival forcing defined uniformly  from parameter $\kappa$, such that $\mathbb Q(\kappa)$ is $\kappa^+$-closed and of cardinality $\leq \kappa^{+\mathcal C}$.  For example, we can take $\mathcal C$ to be the class of inaccessibles and $\mathbb Q(\kappa)$ to be $\add(\kappa^+,1)$ or $\col(\kappa^+,\kappa^{+\mathcal C})$.  Let $\mathbb P(\kappa,\delta)$ be the Easton-support product of 
$\mathbb Q(\alpha)$ over $\alpha\in \mathcal C \cap [\kappa,\delta)$. 
Then $\PP$ is good.

For the following lemma, we don't actually need the factoring condition (\ref{item:factoring})
of goodness:

\begin{lemma}
	\label{gengen}
	Let $\PP$ be good.  Let $j : L \to L$ be elementary.   Suppose $j(\kappa_\alpha) = \kappa_\beta$ and $G \subseteq \mathbb P(\kappa_\alpha,\kappa_{\alpha+1})$ is generic over $L$.  Then $j[G]$ generates a $\mathbb P(\kappa_\beta,\kappa_{\beta+1})$-generic filter over $L$.
\end{lemma}

\begin{proof}
We may and do assume $j(\kappa_{\alpha+n})=\kappa_{\beta+n}$ for all $n<\omega$,
since there is an elementary $k:L\to L$ such that $k\rest\kappa_{\alpha+1}=j\rest\kappa_{\alpha+1}$ and $k(\kappa_{\alpha+n})=\kappa_{\beta+n}$ for $n<\omega$, and hence $j[\mathbb{P}(\kappa_\alpha,\kappa_{\alpha+1})]=k[\mathbb{P}(\kappa_\alpha,\kappa_{\alpha+1})]\sub\mathbb{P}(\kappa_\beta,\kappa_{\beta+1})$. 

Clearly $j[G]$ generates a filter over $\mathbb{P}(\kappa_\beta,\kappa_{\beta+1})$.
So let $A \subseteq \mathbb P(\kappa_\beta,\kappa_{\beta+1})$ be a maximal antichain in $L$.
It suffices to see that $j[G]$ has some element below a condition in $A$.

Fix a Skolem term $\tau$, ordinals ${i_0} < \dots < i_{\ell-1}<\beta$ and a number $n<\omega$ with
	$$A = \tau^L(\kappa_{i_0},\dots,\kappa_{i_{\ell-1}},\kappa_\beta,\kappa_{\beta+1},\dots,\kappa_{\beta+n+1}).$$
	For $s \in [\kappa_\alpha]^\ell$, let 
	$$f(s) = \tau^L(s ^\frown \la \kappa_\alpha,\dots,\kappa_{\alpha+ n +1}\ra)$$
	if this is a maximal antichain in $\mathbb P(\kappa_\alpha,\kappa_{\alpha+1})$, and $f(s) = \mathbb P(\kappa_\alpha,\kappa_{\alpha+1})$ otherwise.  By the distributivity condition (\ref{item:distrib}) of goodness, $L$ has a maximal antichain $B \subseteq \mathbb P(\kappa_\alpha,\kappa_{\alpha+1})$ that refines $f(s)$ for all $s \in [\kappa_\alpha]^\ell$.  Let $p^* \in B \cap G$. Now $j(B)$ refines $j(f)(s)$ for all $s\in[\kappa_\beta]^{\ell}$, and in particular, refines $j(f)(\kappa_{i_0},\ldots,\kappa_{i_{\ell-1}})=A$. Therefore $j(p^*)\in j[G]$
	is below some condition in $A$, as desired.
\end{proof}

\begin{lemma}\label{lem:kappa-cc}
	Let $\PP$ be good and let $\kappa$ be an $L$-indiscernible.
	Then $L\sats$``$\PP(0,\kappa)$ is $\kappa$-cc''.
	\end{lemma}
\begin{proof}
	We may assume $\kappa=\kappa_\om$, by indiscernibility.
	Let $A\subseteq\PP(0,\kappa_\om)$ be a maximal antichain in $L$.
	Then there is a Skolem term $\tau$ and $m,n<\omega$ such that
	\[ A=\tau^L(\kappa_0,\ldots,\kappa_m,\kappa_\omega,\ldots,\kappa_{\omega+n}).\]
	Let $j:L\to L$ be elementary with $\crit(j)=\kappa_{m+1}$ and $j(\kappa_{m+i+1})=\kappa_{\om+i}$ for $i<\omega$. Then $A\in\rg(j)$.
	Let $j(\bar{A})=A$. Note that $\bar{A}=A\cap L_{\kappa_{m+1}}=A\cap\PP(0,\kappa_{m+1})$.
	And by elementarity, $\bar{A}$ is a maximal antichain of $\PP(0,\kappa_{m+1})$.
	But $\PP(0,\kappa_{m+1})\times\PP(\kappa_{m+1},\kappa_\om)\cong\PP(0,\kappa_\om)$
	as witnessed by $\pi$ as in the factoring condition.
	It follows that $\bar{A}=A$, which suffices.
	\end{proof}

\subsection{Two mutually embeddable models}
We first prove part (\ref{item:mutual_pair}) of Theorem \ref{sec:0^sharp}. Let $\PP$ be good.
Let $j : L \to L$ be the minimal nontrivial embedding from $L$ to itself, the one that sends $\kappa_n$ to $\kappa_{n+1}$ for finite $n$, and sends all other $\kappa_i$ to themselves.
Let $G_0$ be 
$\mathbb P(0,\kappa_0)$-generic over $L$, and let $G_1$ and $H_1$ be mutually 
$\mathbb P(\kappa_0,\kappa_1)$-generic, so that $L[G_1] \nsubseteq L[H_1] 
\nsubseteq L[G_1]$.  Since $\mathbb P(\kappa_0,\kappa_1)$ adds no subsets of $\mathbb P(0,\kappa_0)\sub L_{\kappa_0}$, $G_1$ and $H_1$ are both generic over 
$L[G_0]$.

Using Lemma \ref{gengen}, let $G_{n+1}$ be the filter generated by $j[H_n]$, and let 
$H_{n+1}$ be 
the filter generated by  $j[G_n]$.  Each is $\mathbb 
P(\kappa_{n},\kappa_{n+1})$-generic over $L$.

By goodness,
we can identify $\PP(0,\kappa_n)$ with $\PP(0,\kappa_0)\times\PP(\kappa_0,\kappa_1)\times\ldots\times\PP(\kappa_{n-1},\kappa_n)$, and $\PP(0,\kappa_\om)$ with the direct limit of these partial orders
(for the latter, consider embeddings $k:L\to L$ with $\crit(k)=\kappa_n$
and $k(\kappa_n)=\kappa_\om$).
So let
$$G = \{ p_0\!^\frown p_1 \!^\frown \dots \!^\frown p_n : n \in \omega 
\wedge \forall i \leq n (p_i \in G_i) \},$$ and let $$H = \{ p_0\!^\frown p_1 
\!^\frown \dots \!^\frown p_n : n \in \omega \wedge p_0 \in G_0 \wedge \forall i 
\leq n ( i >0 \rightarrow p_i \in H_i) \}.$$ 
For each $n<\omega$, $G \cap 
\mathbb P(0,\kappa_n)$ is generic over $L$, since the filters $G_i$ are mutually 
generic by distributivity and cardinality considerations.
The same holds for 
$H$.

Since $L\sats$``$\PP(0,\kappa_\om)$ is $\kappa_\om$-cc'' by Lemma \ref{lem:kappa-cc},
it follows that $G$ and $H$ are in fact $\PP(0,\kappa_\om)$-generic over $L$.
By construction, $j[G] \subseteq H$ and 
$j[H] \subseteq G$.  Thus the embedding can be lifted in both directions.
This completes clause (a) of part (\ref{item:mutual_pair}) of the theorem.

Now consider clause (b).
Suppose further that $L\sats$``$\PP(\kappa_\alpha,\kappa_\beta)$ is $\kappa_\alpha^+$-closed'' for all  $\alpha<\beta$; note this holds in the
example of a good $\PP$ given earlier. Then note that $\PP^2$ is good. Choose $(G_1,H_1)$ earlier such that it is $\PP(\kappa_0,\kappa_1)^2$-generic over $L$, and hence over $L[G_0]$. By goodness applied to $\PP^2$, each $(G_{n+1},H_{n+1})$ is then $\PP(\kappa_n,\kappa_{n+1})^2$-generic over $L$, and as before,
$(G,H)$ is $\PP(0,\kappa_\om)^2$-generic. This gives that $L[G],L[H]$ are mutually 
generic over $L[G_0]$, and note that $j,k$ extend uniquely elementarily to $j^+:L[G,H]\to L[G,H]$ 
with $j^+(G,H)=(H,G)$ as desired.

Since $0^\sharp$ is definable and yields a definable enumeration
of all the relevant objects in ordertype $\om$, the constructions
above can moreover be done definably.

\subsection{Arbitrarily many models}

To construct larger systems of mutually embeddable models under the assumption that $0^\sharp$ exists, we must first show the existence of the desired filters that are generic over $L$ for uncountable partial orders.
This is handled by the following, which is proven in the same way as Proposition 4.9 in \cite{constructibility_and_class_forcing}.

\begin{lemma}
\label{genexist}
Suppose $0^\sharp$ exists, $\kappa_i$ is the $i^{th}$ Silver indiscernible, and $\mathbb P \in L$ is a partial order such that $L \models \mathbb P$ is $\kappa_i^+$-distributive and of size $< \kappa_{i+\omega_1}$.  Then there is an $(L,\mathbb P)$-generic filter.
\end{lemma}

Suppose $\PP$ is good.  Given any finite sequence of ordinals $i_0<\dots<i_n$, there is a canonical isomorphism $\sigma : \PP(0,\kappa_{i_0}) \times \dots \times \PP(\kappa_{i_{n-1}},\kappa_{i_n}) \to \PP(0,\kappa_{i_n})$ given by the following procedure:  Take the $L$-least isomorphism $\pi : \PP(0,\kappa_{i_{n-1}}) \times \PP(\kappa_{i_{n-1}},\kappa_{i_n}) \to \PP(0,\kappa_{i_n})$ witnessing the factoring condition (\ref{item:factoring}), then do the same for $\PP(0,\kappa_{i_{n-1}})$, repeat down to $\PP(0,\kappa_{i_1})$, and compose.  

By the above lemma, for every successive pair of Silver indiscernibles $\kappa_i<\kappa_{i+1}$, there exists a filter $G_i$ that is $\mathbb P(\kappa_i,\kappa_{i+1})$-generic over $L$.  To build an $(L,\PP(\mu,\kappa))$-generic filter for arbitrary indiscernibles $\mu<\kappa$, we use the following lemma.  Let us use the notational convention that $\PP(\kappa,\kappa)$ is the trivial forcing. 

\begin{lemma}
\label{buildbiggen}
Suppose $\kappa$ is a Silver indiscernible and $p \in \PP(0,\kappa)$.  Then there is a finite sequence of ordinals $i_0<\dots<i_n$ such that $\kappa = \kappa_{i_n}$, and if 
$$\sigma : \PP(0,\kappa_0) \times \PP(\kappa_0,\kappa_{i_0}) \times \prod_{0\leq m <n} \left( \PP(\kappa_{i_m},\kappa_{(i_m)+1}) \times \PP(\kappa_{(i_m)+1},\kappa_{i_{(m+1)}}) \right)
\to \PP(0,\kappa_{i_n})$$
is the canonical isomorphism, then there is a sequence $\la p_{-1},p_0,\dots,p_{n-1} \ra$, with $p_{-1} \in \PP(0,\kappa_0)$ and $p_m \in \PP(\kappa_{i_m},\kappa_{i_m+1})$, such that $p = \sigma(p_{-1},1,p_0,1,\dots,p_{n-1},1)$.

Furthermore, suppose that $G_{-1}$ is $(L,\PP(0,\kappa_0))$-generic and for every ordinal $i$, $G_i$ is $(L,\PP(\kappa_i,\kappa_{i+1}))$-generic.  Let $G$ be the class of $p$ such that for some sequence $\la p_{-1},p_0,\dots,p_{n-1} \ra$ with $p_i \in G_i$, $p = \sigma(p_{-1},1,p_0,1,\dots,p_{n-1},1)$ as above.  Then $G$ is a $\PP(0,\ord)$-generic filter over $L$.
\end{lemma}

\begin{proof}
The first claim is established by induction.  If it is true for $\kappa_i$, then it easily follows for $\kappa_{i+1}$.  Suppose it holds for all $i<\lambda$, where $\lambda$ is a limit.  For $p \in \PP(0,\kappa_\lambda)$, let $i$ be the least such that $p \in \PP(0,\kappa_i)$.  Then if $\pi : \PP(0,\kappa_i) \times \PP(\kappa_i,\kappa_\lambda) \to \PP(0,\kappa_\lambda)$ is the isomorphism, then $p = \pi(p,1)$, and the desired decomposition holds by induction.

For the second claim, we show by induction that $G \cap \PP(0,\kappa)$ is generic over $L$ for every indiscernible $\kappa$.  If this is true for $\kappa_i$, then since $\PP(0,\kappa_i) \subseteq L_{\kappa_i}$ and $\PP(\kappa_i,\kappa_{i+1})$ is $\kappa_i^+$-distributive, $G \cap\PP(0,\kappa_i)$ is generic over $L[G_i]$, and thus $(G \cap \PP(0,\kappa_i))\times G_i$ is generic.  For limit $i$, genericity follows by the $\kappa_i$-c.c., established in Lemma \ref{lem:kappa-cc}.
\end{proof}

%
%

We want to build a system of $2^\Omega$-many mutually embeddable models, where $\Omega = \ord$.  Formally, we should think of the construction as either working up to a strong limit cardinal $\Omega$, or as being carried out with the help of a global choice function and producing a correspondence between class functions $f \in 2^\Omega$ and models of the form $L[G_f]$, which will be mutually embeddable.  We will choose generics for the forcings $\mathbb P(\kappa_i,\kappa_{i+1})$, either using Lemma \ref{genexist}, or by taking images of such generics using Lemma \ref{gengen}.  We will additionally assume moreover that for every pair of indiscernibles $\mu<\kappa$, $\PP(\mu,\kappa)^2$ is $\mu^+$-distributive in $L$.

Recursively define a class function $F : 2^{<\Omega} \to \Omega^{<\Omega}$ with the following properties:
\begin{enumerate}
 \item $\dom(F(\sigma))=\dom(\sigma)$, 
 \item $F(\sigma)$ is strictly increasing,
\item For every strong limit cardinal $\lambda$, $\bigcup_{\sigma \in 2^{<\lambda}} \ran(F(\sigma))$ is the set of odd ordinals below $\lambda$.  ($\alpha$ is \emph{odd} when it is of the form $\beta+2n+1$, where $\beta$ is a limit ordinal and $n<\omega$.  Otherwise, it is \emph{even}.)
\item if 
$\sigma\sub\tau$ then  $F(\sigma)\sub F(\tau)$,
and
\item if $\sigma\not\sub\tau\not\sub\sigma$
and $\alpha$ is least such that $\sigma(\alpha)\neq\tau(\alpha)$
then
\[ \ran(F(\sigma))\inter\ran(F(\tau))=F(\sigma)[\alpha]=F(\tau)[\alpha].\]
\end{enumerate}

Let $H \subseteq \mathbb P(0,\kappa_0)$ be generic over $L$.  Next, construct a sequence of pairs of generic filters, $\la (G_\alpha^0,G_\alpha^1) : \alpha \in \Omega \ra$.  If $\alpha$ is even, let $G_\alpha^0,G_\alpha^1$ be mutually $\mathbb P(\kappa_\alpha,\kappa_{\alpha+1})$-generic over $L$.  This is possible because of Lemma \ref{genexist} and the assumption that $\mathbb P(\kappa_\alpha,\kappa_{\alpha+1})^2$ is $\kappa_\alpha^+$-distributive in $L$.

For the generics at odd levels, we take images of previously chosen generics under certain elementary embeddings.  For $\sigma \in 2^{<\Omega}$, let $j_\sigma : L \to L$ be the elementary embedding generated by mapping $\kappa_i$ to $\kappa_{F(\sigma)(i)}$ for $i \in \dom(\sigma)$, and mapping $\kappa_i$ to the least possible indiscernible for larger $i$.  If $\alpha$ is an odd ordinal, let $\tau\in 2^{<\Omega}$ be the shortest sequence such that $\alpha \in \ran(F(\tau))$, and let $\beta=\max(\dom\tau)$.  If $\alpha$ is of the form $\lambda +2n+1$, then $\beta \leq \lambda+n<\alpha$.  Let $G^0_\alpha = G^1_\alpha$ be the filter generated by $j_\tau[G_\beta^{\tau(\beta)}]$.  It follows from Lemma \ref{gengen} that this filter is $\mathbb P(\kappa_\alpha,\kappa_{\alpha+1})$-generic over $L$.

For any class function $f : \Omega \to 2$, we get, using Lemma \ref{buildbiggen}, a class-generic extension $L[G_f]$ of $L$ for the class forcing $\mathbb P(0,\Omega)$, generated by:
$$\{ p_0\!^\frown p_1 \!^\frown \dots \!^\frown p_{n} : p_0 \in H, \text{ and for some }\alpha_1<\dots<\alpha_{n}, p_i \in G_{\alpha_i}^{f(\alpha_i)} \text{ for } 1\leq i \leq n \}$$
If $f \not= g$ are functions that disagree at some even ordinal $\alpha$, then $G_\alpha^{f(\alpha)},G_\alpha^{g(\alpha)}$ are mutually generic, so $L[G_f] \nsubseteq L[G_g]$.

To see that $L[G_f]$ elementarily embeds into $L[G_g]$, let $j_f : L \to L$ be defined by $j_f(x) = \lim_{\alpha \to \Omega} j_{f \rest \alpha} (x)$.  This map lifts to the desired embedding, since for every $\alpha$ and $p \in G_\alpha^{f(\alpha)}$, $j_f(p) \in j_{f \rest( \alpha+1)}[G_\alpha^{f(\alpha)}] \subseteq G_\beta^0=G_\beta^1$, where $\beta$ is the top ordinal in the range of $F(f\rest(\alpha+1))$.

This concludes the proof of Theorem \ref{main0sharp}.
\end{proof}

\begin{remark}
The construction of the tree up to height $\omega$ can be done definably in $L[0^\sharp]$.  The reals of $V$ determine a system of mutually embeddable models.
\end{remark}

\subsection{Forcing ``$V=\HOD$''}

There are many ways to code information while forcing so that the extension satisfies that every set is ordinal-definable.  Many of these already fit into our scheme for producing mutually-embeddable models.  Known mechanisms include coding into the GCH pattern \cite{MR292670}, coding into the $\diamondsuit^*$ pattern \cite{MR2518815}, coding into stationarity \cite{MR2548481}, and coding into a splitting property \cite{MR4013973}, and possibly others.  Let us describe the abstract features that are needed to get mutually embeddable models of ``$V=\HOD$''.

As in the examples discussed before, let $\mathcal C$ be a class of ordinals defined in $L$ and contained in the $L$-inaccessibles.  To simplify notation, let $\theta_\alpha$ denote $(\alpha^{+\mathcal C})^+$, the $L$-successor cardinal of the least ordinal in $\mathcal C$ above $\alpha$.   For $\alpha \in \mathcal C$, we require that $\mathbb Q(\alpha)$ is a nontrivial forcing in $L$ defined from $\alpha$ such that $L\models \mathbb Q(\alpha)$ is $\alpha^+$-closed and of cardinality $\leq \alpha^{+\mathcal C}$.  Furthermore, we require that there is a formula $\varphi(x,y)$ such that for all $\alpha \in \mathcal C$, the following holds:

\begin{itemize}
\item[$(*)$] Whenever $\mathbb S$ is a forcing in $L$ of cardinality $\leq\alpha$, and $G \times H \subseteq \mathbb Q(\alpha) \times \mathbb S$ is generic over $L$, then $G$ is the unique object $y$ such that $L_{\theta_\alpha}[G \times H] \models \varphi(\alpha,y)$.
\end{itemize}
This is made possible by coding $G$ into some information about ordinals between $\alpha$ and $\alpha^{+\mathcal C}$ that cannot be altered by small forcing, such as the GCH pattern.

Let $\mathbb{P}(\kappa,\delta)$ be  the Easton-support product of the $\mathbb Q(\alpha)$'s for $\kappa \leq \alpha < \delta$. Then the extension will satisfy ``$V=\HOD$''.  To prove this, it suffices to show: 
\begin{claim}
If $G \subseteq \mathbb P(\kappa,\delta)$ is generic over $L$, then $G$ is definable in $L[G]$.
\end{claim}

\begin{proof}
For $\kappa \leq \alpha < \beta \leq \delta$, let $G_{\alpha,\beta} = G \cap \mathbb P(\alpha,\beta)$, and let $G_\alpha = G_{\alpha,\alpha+1}$.  Note that for $\alpha \in \mathcal C \cap [\kappa,\delta)$, $\mathbb P(\kappa,\delta) \cong \mathbb P(\kappa,\alpha) \times \mathbb Q(\alpha) \times \mathbb P(\alpha+1,\delta)$.  By our assumption $(*)$, $G_\alpha$ is the unique object $y$ such that $L_{\theta_\alpha}[G_{\kappa,\alpha+1}] \models \varphi(\alpha,y)$.  By the distributivity of $\mathbb P(\alpha+1,\delta)$, $L_{\theta_\alpha}[G_{\kappa,\alpha+1}] = H_{\theta_\alpha}^{L[G]}$.  Thus $G_\alpha$ is definable in $L[G]$ as the unique $y \in H_{\theta_\alpha}$ such that $H_{\theta_\alpha} \models \varphi(\alpha,y)$.  By the uniformity of these definitions, $G$ is definable in $L[G]$ in terms of the parameter $\delta$.  But because $L[G]$ has no generics over $L$ for nontrival forcings of distributivity greater than $\delta$, $G$ is moreover definable without parameters in $L[G]$.
\end{proof}

The desired result can be also be obtained by a slightly different method.  To force ``$V=\HOD$'' over an arbitrary model of $\ZFC$, we typically code information into properties of ordinals via an iteration rather than a product.  This can also be subsumed under the present scheme.  For example, suppose $\mathcal C$ is the class of inaccessibles, and we recursively code every set into the GCH pattern via iteration, with the iteration only active at successor cardinals.  Then a standard termspace argument (see \cite{MR2768691}) shows that, if the length of the iteration is $\delta$, then this iteration is a projection of the Easton-support product of $\col(\alpha^+,\alpha^{+\mathcal C})$, over $\alpha \in \mathcal C \cap\delta$.  Let us call the above iteration $\mathbb A$ and the above Easton-support product $\mathbb B$. Starting from $L$ or a similar ground model, we can ensure that $\mathbb A$ is definable from $\delta$, and we can take a projection $\pi : \mathbb B \to \mathbb A$ that is likewise definable.

Now we have already argued that, if $\delta = \kappa_\omega$, then we have a system of generic extensions $\la L[G_r] : r \in \mathbb R \ra$, pairwise $\subseteq$-incomparable and mutually embeddable, where each $G_r$ is $\mathbb B$-generic over $L$.  If $j : L[G_r] \to L[G_s]$ is a witnessing embedding, then because $\pi$ is defined in $L$ from the parameter $\delta$ which is fixed by $j$, $j(\pi) = \pi$.  For each $r \in \mathbb R$, let $H_r = \pi[G_r]$.  Then for every $r,s \in \mathbb R$ and every $a \in H_r$, if $b \in G_r$ is such that $\pi(b) = a$, then $j(a) = j(\pi(b)) = j(\pi)(j(b))= \pi(j(b)) \in \pi[G_s] = H_s$.  Thus the embeddings restrict to the submodels $L[H_r]$, which satisfy ``$V=\HOD$''.  In summary, we have:

\begin{theorem}\label{cor:hod-2-cycle}
The following are equivalent to the existence of $0^\sharp$:
\begin{enumerate}
\item There are two definable $\subseteq$-incomparable proper class models $M,N$, and definable elementary embeddings $j : M \to N$, $k : N \to M$, such that $j \restriction \ord = k \restriction \ord$, and each model satisfies ``$V = \HOD$''.
\item There is a definable system of $\subseteq$-incomparable mutually embeddable proper class models $\la M_r : r \in \mathbb R \ra$, with each model satisfying  ``$V = \HOD$''.
\end{enumerate}
\end{theorem}

\section{Constructions using iterated ultrapowers}
\label{sec:measurables}

In this section, we show that if there are measurably-many measurables, then we can construct large systems of mutually embeddable distinct models, all of which are certain iterated ultrapowers of $V$.

\begin{theorem}
Let $\kappa$ be a measurable cardinal, and let $\langle \mu_i : i < \kappa\rangle$ be a sequence of measurable cardinals above $\kappa$. 
Then for every infinite cardinal $\nu$, there is a collection of $\nu$-many distinct mutually embeddable well-founded models, each having the same theory as $V$.
\end{theorem}

\begin{proof}
We give a construction of a system as claimed of size $2^\nu$ for $\nu\leq\kappa$.  In order to show the general claim, let $\nu$ be arbitrary, and let $j : V \to M$ be an iterated ultrapower embedding by a measure on $\kappa$, with the iteration long enough such that $j(\kappa) \geq \nu$.  Then apply the construction in $M$.

Let $\mathcal U$ be a normal measure on $\kappa$, and let $\vec{\mathcal V} = \la\mathcal V_i : i < \kappa \ra$ be such that for each $i<\kappa$, $\mathcal V_i$ is a normal measure on $\mu_i$ of Mitchell-order zero.  For each $f \in 2^\nu$, we define an iteration of ultrapowers,
$\la M_i^f : i < \nu \ra$ and let $M_f$ be the direct limit.

Let us describe the construction of the model associated to a given binary sequence $f$ of any length, while suppressing the superscript.  Let $M_0 = V$.  For $\alpha > 0$, assume we have a commuting system of embeddings $\la j_{\gamma,\beta} : M_\gamma \to M_\beta \mid 0 \leq \gamma \leq \beta < \alpha \ra$.  If $\alpha$ is a limit ordinal, let $M_\alpha$ be the direct limit.  
If $\alpha = \beta+1$, let $i_\alpha : M_\beta \to M^*_{\alpha}$ be the ultrapower map by $j_{0,\beta}(\mathcal U)$.
If $f(\beta) =0$, let $M_{\alpha} = M^*_{\alpha}$.  
Otherwise, let $M_{\alpha} = \ult(M^*_{\alpha},i_\alpha \circ j_{0,\beta}(\vec{\mathcal V})(j_{0,\beta}(\kappa)))$.  
Extend the system of maps accordingly.

Let $f,g$ be distinct binary sequences of the same length $\lambda \leq \kappa$, and let $\alpha$ be the first ordinal for which $f(\alpha) \not= g(\alpha)$; say $f(\alpha) = 0$.  Then $M^f_\beta = M^g_\beta$ for $\beta \leq \alpha$, and the systems of maps are likewise the same up to this point.
Let $\delta = i_\alpha \circ j_{0,\alpha}(\la \mu_i : i < \kappa \ra)(j_{0,\alpha}(\kappa))$, and let $\mathcal V^* = j_{0,\alpha+1}(\vec{\mathcal V})(\delta)$.  
Then $\delta$ is measurable in $M^*_\alpha = M^f_{\alpha+1}$, but because $\mathcal V^*$ has Mitchell-order zero, $\delta$ is not measurable in $M^g_{\alpha+1}$.  Since $f,g\in M^*_\alpha$, $\delta$ is inaccessible in both models, we only apply measures strictly below or above $\delta$ in the further sequence of embeddings according to either $f$ or $g$, and the length of the iterations is $\lambda < \delta$, we have that $\delta = j^f_{\alpha+1,\lambda}(\delta) = j^g_{\alpha+1,\lambda}(\delta)$.  Thus $M_f \not= M_g$.

Let us first prove the claim for $\nu = \omega$.
We claim that if $f,g \in 2^\omega$ and $g$ takes both values 0 and 1 infinitely often, then $M_f$ is embeddable into $M_g$.  Let us fix some additional notation.  Let $\mu = \sup_{i<\kappa} \mu_i$.  For a given $f \in 2^\omega$, define:
$$\kappa^f_n = j^f_{0,n}(\kappa); \,\, \delta^f_n =  j^f_{0,n+1}(\la \mu_i : i < \kappa \ra)(\kappa^f_n).$$
Every element of $M^f_{n+1}$ is of the form $j^f_{0,n+1}(h)(\kappa^f_0,\delta^f_0,\dots,\kappa^f_n,\delta^f_n)$, where $h$ is a function with domain $(\kappa \times \mu)^{n+1}$.  (Note that these ordinals $\kappa_i^f,\delta_i^f$ are fixed by later embeddings.)

Let $s$ be an increasing function on  $\omega \cup \{ -1\}$ such that $s(-1) = -1$ and for every $n \geq 0$, $f(n) = g(s(n))$.  For $n \geq 0$, if $p_n$ is the finite sequence 
$$\la g(s(n-1)+1),\dots,g(s(n)-1) \ra,$$
 then we can write $g$ as:
$$p_0 \,^\frown \la f(0) \ra ^\frown p_1 \,^\frown \la f(1) \ra ^\frown  \dots ^\frown  p_n \,^\frown \la f(n) \ra ^\frown\dots$$

Suppose $M_f \models \varphi(x)$.  Then for some $n$ and some function $h$ on $(\kappa \times \mu)^{n+1}$, $x = j^f_{0,\omega}(h)(\kappa^f_0,\delta^f_0,\dots,\kappa^f_n,\delta^f_n)$.  Furthermore, this is equivalent to 
$$M^f_{n+1} = M_{f \restriction (n+1)} \models \varphi\left(j^f_{0,n+1}(h)(\kappa^f_0,\delta^f_0,\dots,\kappa^f_n,\delta^f_n)\right).$$
Let $g_0 = p_0 \,^\frown f \restriction (n+1)$, and let $m_0 = \len(p_0)$.  By taking the sequence of ultrapowers of $V$ determined by $p_0$ and applying elementarity, we have:
$$M_{g_0} \models \varphi \left(j^{g_0}_{0,m_0+n+1}(h)(\kappa^{g_0}_{m_0},\delta^{g_0}_{m_0},\kappa^{g_0}_{m_0+1},\delta^{g_0}_{m_0+1},\dots,\kappa^{g_0}_{m_0+n},\delta^{g_0}_{m_0+n}) \right).$$
This is expressible in $M_{p_0 \,^\frown \la f(0) \ra}$ in terms what is satisfied in the iterated ultrapower following $\la f(1),\dots,f(n) \ra$.   So we can take the series of ultrapowers following $p_1$ over $M_{p_0 \,^\frown \la f(0) \ra}$ to insert $p_1$ in the appropriate place.  Let $m_1 = \len(p_1)$, and let $g_1 = p_0 \,^\frown \la f(0) \ra ^\frown p_1 \,^\frown \la f(1),\dots,f(n) \ra$.  By elementarity, we get:
\begin{align*}
M_{g_1} \models \varphi (j^{g_1}_{0,m_0+m_1+n+1}(h)	& (\kappa^{g_1}_{m_0},\delta^{g_1}_{m_0}, \kappa^{g_1}_{m_0+m_1+1},\delta^{g_1}_{m_0+m_1+1},\dots \\
											&\dots,\kappa^{g_1}_{m_0+m_1+n},\delta^{g_1}_{m_0+m_1+n}) ).
\end{align*}
We continue in this way with a sequence of functions $g_0,g_1,\dots,g_n$, where
$g_n = p_0 \,^\frown \la f(0) \ra ^\frown p_1 \,^\frown \la f(1) \ra ^\frown  \dots ^\frown  p_n \,^\frown \la f(n) \ra.$
We have:
\begin{align*}
M_{g_n} \models \varphi (	& j^{g_n}_{0,\sum_{i=0}^n m_i+n+1}(h)	(\kappa^{g_n}_{m_0},\delta^{g_n}_{m_0}, \kappa^{g_n}_{m_0+m_1+1},\delta^{g_n}_{m_0+m_1+1}, \\
					&\kappa^{g_n}_{m_0+m_1+m_2+2},\delta^{g_n}_{m_0+m_1+m_2+2},\dots,\kappa^{g_n}_{\sum_{i=0}^n m_i+n},\delta^{g_n}_{\sum_{i=0}^n m_i+n}) ),
\end{align*}
where $m_i = \len(p_i)$.  This is equivalent to:
$$M_{g} \models \varphi \left(j^{g}_{0,\omega}(h) (\kappa^{g}_{s(0)},\delta^{g}_{s(0)}, \dots, \kappa^{g}_{s(n)},\delta^{g}_{s(n)}) \right).$$
Therefore, we may define an elementary embedding $k_{f,g} : M_f \to M_g$ by:
 $$j^f_{0,\omega}(h)(\kappa^f_0,\delta^f_0,\dots,\kappa^f_n,\delta^f_n) \mapsto j^g_{0,\omega}(h)(\kappa^g_{s(0)},\delta^g_{s(0)},\dots,\kappa^g_{s(n)},\delta^g_{s(n)}).$$

For the general case, we can similarly show that if $f,g \in 2^\nu$ are such that $g$ takes both values 0 and 1 unboundedly often, then $M_f$ is embeddable into $M_g$.  This uses the fact that in longer iterations, objects in direct limits are still found by applying the image of a function from $V$ to a finite sequence of critical ordinals.  For suppose $M_f \models \varphi(x)$ and $x = j^f_{0,\nu}(h)(\kappa^f_{\alpha_0},\delta^f_{\alpha_0},\dots,\kappa^f_{\alpha_n},\delta^f_{\alpha_n})$.  Write $f$ as
$$q_0 \,^\frown \la f(\alpha_0) \ra ^\frown q_1 \,^\frown \la f(\alpha_1) \ra ^\frown  \dots ^\frown  q_n \,^\frown \la f(\alpha_n) \ra ^\frown\dots$$
If $t = \la f(\alpha_0),\dots,f(\alpha_n) \ra$, then similar arguments as above show that the $q_i$'s can be inserted over $M_t$ to show that
$$M_{t} \models \varphi \left( j_{0,n+1}^{t}(h)(\kappa^{t}_0,\delta^t_0,\dots,\kappa^t_n,\delta^t_n) \right).$$
If $s : \nu \to \nu$ is an increasing function such that $f = g \circ s$, then we can insert finitely many subsequences of $g$ over $M_t$ to get:
$$M_{g} \models \varphi \left(j^{g}_{0,\nu}(h) (\kappa^{g}_{s(\alpha_0)},\delta^{g}_{s(\alpha_0)}, \dots, \kappa^{g}_{s(\alpha_n)},\delta^{g}_{s(\alpha_n)}) \right).$$
Thus, the map defined like before is elementary.
\end{proof}

\section{Strong 2-cycles}
\label{sec:2-cycles}
\begin{definition}[\cite{gaifman},\cite{strong_axioms}]
We say that $\mathrm{I3}(\lambda)$ holds if there is a non-trivial elementary embedding $j \colon V_\lambda \to V_\lambda$, where $\lambda$ is a limit ordinal.    
\end{definition}
The axiom $\exists \lambda \mathrm{I3}(\lambda)$ is a very strong large 
cardinal axiom; see \cite{the_higher_infinite} for more details.

Suppose $\mathrm{I3}(\lambda)$ holds and fix a nontrivial elementary embedding $j : V_\lambda \to V_\lambda$.  
Let $\kappa$ be the critical point of $j$, and let $\mathcal U$ be the normal measure on $\kappa$ derived from $j$.  That is, $X \in \mathcal U$ iff $X \subseteq \kappa$ and $\kappa \in j(X)$. Let $M = \ult(V_\lambda,\mathcal U)$, and let $i : V_\lambda \to M$ be the ultrapower embedding.  Of course, $M\not=V_\lambda$.  $M$ can be embedded back into $V_\lambda$ by a standard application of \L o\'s' Theorem: Let $k : M \to V_\lambda$ be defined by $ k([f]_{\mathcal U}) = j(f)(\kappa)$. Moreover, even though $\lambda$ is a strong limit singular cardinal of countable cofinality, $V_\lambda \models \ZFC$.  We call a pair of mutually embeddable distinct models a 2-cycle.
Thus the axiom I3 provides a relatively straightforward example of a 2-cycle. 

Furthermore, this argument easily produces many different transitive submodels of $V_\lambda$ that are mutually embeddable.  Working with a fixed $j$, we can derive various kinds of ultrapowers that embed back into $V_\lambda$ for the same reason as above, and are thus mutually embeddable by compositions of embeddings.  We can also take advantage of the fact that there are many self-embeddings of $V_\lambda$ with various critical points, since for any two elementary $j,k : V_\lambda \to V_\lambda$, the function $\bigcup_{\alpha<\lambda} j(k\restriction V_\alpha)$ is an elementary embedding.

However, we do not need assumptions as strong as I3 in order to have such 2-cycles.  
Suppose $E$ is an extender.  Let $M_0 = V$ and inductively define the embeddings $j_{n,n+1} : M_n \to M_{n+1} = \ult(M_n,j_{0,n}(E))$ and $j_{m,n+1} =j_{n,n+1} \circ j_{m,n}$ for $m < n$.  By standard arguments, the direct limit $M_\omega$ is well-founded thus can be identified with its transitive collapse.  Applying $j_{0,1}$ to the directed system defining $M_\omega$ yields the same system, only with the first model removed.  Thus $j_{0,1} \restriction M_\omega$ is a self-embedding of $M_\omega$.  
In order to obtain a non-trivial 2-cycle, we would like to find a transitive model $N$ such that the elementary embedding $j_{0,1} \restriction M_{\omega}$ splits through $N$. 

Let us assume that $E$ is an extender such that its derived ultrafilter $\mathcal U$ over its critical point $\kappa$ is a member of $\ult(V,E)$. Then $\mathcal U \in M_\omega$, and inside $M_\omega$ we can compute $N = \ult(M_\omega,\mathcal U) \subsetneq M_\omega$.  For functions $f_1,\dots,f_n : \kappa \to M_\omega$ that live in $M_\omega$ and a formula $\varphi(x_1,\dots,x_n)$, we have 
$$\{ \alpha : M_\omega \models \varphi(f_1(\alpha),\dots,f_n(\alpha)) \} \in \mathcal U \Longleftrightarrow M_\omega \models  \varphi(j_{0,1}(f_1)(\kappa),\dots,j_{0,1}(f_n)(\kappa)).$$
Thus $N$ can be embedded back into $M_\omega$ as before.
An extender as above exists if there is cardinal $\kappa$ that is 
+2-strong.  But by restricting a +2-strong extender just to the length 
needed to capture its derived ultrafilter, we see that the consistency strength 
of having an extender with this property is less than +2-strongness.  But it is 
still significant.  Recall that if $\kappa$ is measurable then 
$o(\kappa)\leq(2^\kappa)^+$:

\begin{theorem}
Suppose $M$ is a transitive model of ZFC, $j \colon M \to M$ is an elementary embedding with critical point $\kappa$, $\mathcal U$ is 
the normal ultrafilter on $\kappa$ derived from $j$ with seed $\kappa$, and 
$\mathcal U \in M$.  Then:
\begin{enumerate}
 \item\label{item:M_models_o(kappa)_large}  $M\models o(\kappa) =
(2^\kappa)^+$,
 \item\label{item:M_models_X_in_ult} $M\models\text{for every 
}X\subseteq\p(\kappa)$ there is a normal measure $\mu$  on $\kappa$ such that 
$X\in\ult(V,\mu)$,
 \item\label{item:reflect_measure_one} for $U$-measure one many $\alpha<\kappa$,
 $M\models$ the two statements above regarding $\alpha$ instead of $\kappa$, 
and in fact,
 \item\label{item:reflect_varphi_measure_one} for every formula $\varphi$ 
and every $A\in\p(\kappa)\cap M$, 
\[ M\models\varphi(A)\iff[M\models\varphi(A\cap\alpha)\text{ for 
}U\text{-measure one many }\alpha].\]
\end{enumerate}
\end{theorem}

\begin{proof}
Let $i : M \to N = \ult(M,\mathcal U)$ be the ultrapower embedding, and let $k 
: N \to M$ be defined by $k([f]_{\mathcal U}) = j(f)(\kappa)$. 
Now $\kappa 
\in \ran(k)$ and $\p(\kappa)^M=\p(\kappa)^N$
and $N\subseteq M$.
Note $\kappa<\crit(k)$,
so $k(X)=X$ for all $X\in\p(\p(\kappa))^N$,
which implies that $\theta=(2^\kappa)^N=(2^\kappa)^M$
and $\theta<\crit(k)$.
Also $(\theta^+)^N<i(\kappa)<(\theta^+)^M$,
and therefore $\crit(k)=(\theta^+)^N$
and $k((\theta^+)^N)=(\theta^+)^M$.

Let $\gamma = o(\kappa)^M$, let $\xi = o(\kappa)^N$, and note 
that $\xi< \gamma$.  If $N \models \xi <(2^\kappa)^+$, then  $k(\xi)=\xi$,
so $M \models o(\kappa) = \xi$, a contradiction.  Thus $N \models o(\kappa) 
=(2^\kappa)^+$.
Applying $k$ therefore gives
$M\models o(\kappa)= (2^\kappa)^+$,
yielding part
(\ref{item:M_models_o(kappa)_large}).

For part (\ref{item:M_models_X_in_ult}), again since $\crit(k)>\kappa$,
it suffices to see that $N$ satisfies the same statement about $\kappa$.
So let $X\in N$ with $X\subseteq\p(\kappa)$.
Then $k(X)=X$ (and $k(\kappa)=\kappa$),
so it suffices to see that $M\models$``there is a normal measure
$\mu$ on $\kappa$ such that $X\in\ult(V,\mu)$'' (instead of $N$).
But $\mu=\mathcal{U}$ witnesses this in $M$.

For parts \ref{item:reflect_measure_one} and  
\ref{item:reflect_varphi_measure_one}, we have $M\models\varphi(A)$ iff
$N\models\varphi(A)$ because $\crit(k)>\kappa$,
and $N\models\varphi(A)$ iff
$M\models\varphi(A\cap\alpha)$ for $U$-measure one many $\alpha<\kappa$
as usual.
\end{proof}

We can use this method to obtain a large collection of mutually embeddable $\subseteq$-incomparable models, assuming there is $\kappa$ which is +2-strong and such that $2^{\kappa^+} > \kappa^{++}$.  This can be forced from a +3-strong cardinal by standard techniques.

\begin{theorem} 
 Suppose $\kappa$ is $+2$-strong, $2^\kappa = \kappa^+$, and $2^{\kappa^+} > 
\kappa^{++}$.  Then there is a collection of $\kappa^{+3}$-many pairwise 
$\subseteq$-incomparable mutually embeddable proper class transitive models of 
ZFC. Moreover, we can arrange that
the models all contain $\p(\kappa)$
but disagree pairwise over $\p(\kappa^+)$,
and the witnessing embeddings all have critical point $\kappa$.
\end{theorem}

\begin{proof}
Let $\kappa$ be as hypothesized, and let $E$ be an extender witnessing that 
$\kappa$ is +2-strong.  If $\lambda = 2^{\kappa^+}$, then we may assume that $E$ 
is a $(\kappa,\lambda)$-extender such that $\p(\kappa^+) \subseteq M = 
\ult(V,E)$.   Recall that $E$ is collection of ultrafilters $\la \mathcal U_a : 
a \in [\lambda]^{<\omega} \ra$ that project to one another in the appropriate 
sense, and $M$ is the direct limit of the models $M_a=\ult(V_,\mathcal U_a)$.  
If $j_E : V \to M$ is the direct limit map, then for each $a \in 
[\lambda]^{<\omega}$, $\mathcal U_a = \{ X \subseteq [\kappa]^{|a|} : a \in 
j_E(X) \}$.  For each $X \subseteq \kappa^+$, there is $a \in 
[\lambda]^{<\omega}$ such that $X \in M_a$.

Let us iterate $E$ $\omega$-times to obtain a system of models $\la M_i : i 
\leq \omega \ra$ and embeddings $\la j_{\alpha,\beta} : \alpha < \beta \leq 
\omega \ra$ as before.  Again, $j_{0,1} = j_E$ yields a self-embedding of  
$M_\omega$.  Since each $\mathcal U_a \in M_1$ and $\crit(j_{1,\beta}) > 
\lambda$ for $2 \leq \beta \leq \omega$, each $\mathcal U_a \in M_\omega$.  
Working in $M_\omega$, we can define the ultrapower map
\[ j_a^{M_\omega}:M_\omega\to N_a=\ult(M_\omega,\mathcal U_a).\]
Let $j^V_a:V\to M_a$ be
likewise for $V$. Since $V_{\kappa+2}^{M_\omega}=V_{\kappa+2}$,
for each $a\in[\lambda]^{<\omega}$,
 \[ \ult(V,\mathcal U_a)\inter V_{j^V_a(\kappa)+2}=\ult(M_\omega,\mathcal 
U_a)\inter V_{j^{M_\omega}_a(\kappa)+2}.\]
In particular, for each $X \subseteq \kappa^+$, $X \in \ult(M_\omega,\mathcal 
U_a)$ iff $X \in \ult(V,\mathcal U_a)$.  Like before, we can embed each 
$N_a$ back into $M_\omega$ via the factor map $k_a$ defined by 
$k_a([f]_{\mathcal U_a}) = j_E(f)(a)$.

Now since $2^\kappa = \kappa^+$, for each $a$, $\p(\kappa^+)^{N_a}$ has size 
only $\kappa^+$ from the point of view of $V$.  We can choose a sequence $\la 
a_\alpha : \alpha < \kappa^{+3} \ra$ such that for all $\alpha < \kappa^{+3}$, 
there is $X_\alpha \in \p(\kappa^+)^{N_{a_\alpha}}$ which is not in 
$\p(\kappa^+)^{N_{a_\beta}}$ for any $\beta < \alpha$.   For $\eta<\kappa^{+3}$
let
\[ f(\eta)=\sup\{\gamma<\eta : X_\gamma\in N_{a_\eta}\}, \]
and note that if $\cf(\eta)=\kappa^{++}$ then $f(\eta)<\eta$.
Therefore we get a stationary set $Y\sub\kappa^{+3}$
such that $f\rest Y$ is constant, and for $\eta,\xi\in Y$
with $\eta\neq\xi$, we have $X_\eta\notin N_{a_\xi}$
and $X_\xi\notin N_{a_\eta}$. We use $\left<N_{a_\eta}\right>_{\eta\in Y}$
as our desired system of models. For $a,b \in [\lambda]^{<\omega}$,  
\[ j^{M_\om}_b \circ k_a:N_a\to N_b \]
is elementary with 
critical point $\kappa$, so we are done.
\end{proof}

From a weaker assumption, we can achieve a similar situation in a forcing 
extension, along with GCH.  We can also arrange that 
the embeddings between any two of the models in our collection are the same on 
the ordinals. 

\begin{theorem}
Suppose GCH holds and $\kappa$ is $+2$-strong.  
 Then in a forcing extension, 
there is a system  $\la N_\alpha : \alpha < \kappa^{++} \ra$ of 
$\subseteq$-incomparable transitive proper class models of ZFC, along with 
elementary embeddings $\pi_{\alpha,\beta} : N_\alpha \to N_\beta$ such that 
\[ \pi_{\alpha_0,\beta_0} \restriction \ord = \pi_{\alpha_1,\beta_1} \restriction 
\ord \]
whenever $\alpha_0,\alpha_1,\beta_0,\beta_1<\kappa^{++}$ and $\alpha_i \neq \beta_i$ for $i=0,1$.
\end{theorem}

\begin{proof}
Suppose $E$ is a $(\kappa,\kappa^{++})$-extender witnessing that $\kappa$ is 
$+2$-strong, and let $\mathcal U$ be the derived normal ultrafilter over 
$\kappa$.  Iterate $E$ $\omega$-times to obtain a system of models $\la 
M_i : i \leq \omega \ra$ and embeddings $\la j_{\alpha,\beta} : \alpha < \beta 
\leq \omega \ra$ as before.   Let $M = M_\omega$ and note that $j_{0,1} = j_E$ 
yields a self-embedding of $M$.  Let $i : M \to N = \ult(M,\mathcal U)$ be the 
ultrapower embedding, and let $k : N \to M$ be the factor map.  Since $\mathcal 
U \in M$, $(\kappa^{++})^N < i(\kappa) < (\kappa^{++})^M = \kappa^{++}$.  Thus 
$\crit(k) =  (\kappa^{++})^N$.

Let $\mathbb P(\gamma,\delta)$ be the Easton-support product of 
$\add(\alpha^+,\alpha^{++})$ over inaccessible $\alpha\in [\gamma,\delta)$.  Let 
$N_1 = \ult(V,\mathcal U)$.  Now $N_1$ is $\kappa$-closed
and letting $\mathbb Q^{N_1}=\mathbb P(\kappa+1,j_{\mathcal U}(\kappa))^{N_1}$,
then  $N_1\sats$``$\mathbb Q^{N_1}$ is 
$\alpha$-closed'',
where $\alpha$ is the least inaccessible of $N_1$ which is $>\kappa$.
So a standard argument shows that we can build
a filter $K \subseteq \mathbb Q^{N_1}$ which is generic over 
$N_1$. We can moreover do this working in $M_1$,
so we get $K\in M_1$. Let $k_1 : N_1 \to M_1$ be the 
factor map. We will observe in Claim \ref{claim:k_1[K]_gens_filter} below
that $k_1[K]$ generates 
 an $M_1$-generic filter $K'$.
For this, we first need to observe how the dense sets in $M_1$ can
nicely represented, which is just a general fact about factors of ultrapowers:

\begin{claim}
Let $x\in M_1$. Then there is a function $f:[(\kappa^{++})^{N_1}]^{<\om}\to N_1$
and $a\in[\kappa^{++}]^{<\om}$ such that 
$x=k(f)(a)$.\end{claim}
\begin{proof}To see this, let $g\in V$ be a function $g:[\kappa]^{<\om}\to V$
and $a\in[\kappa^{++}]^{<\om}$ 
with $x=j_E(g)(a)$. Then let $f=j(g)\rest[(\kappa^{++})^{N_1}]^{<\om}$.
Since $j_E(g)=k(j_{\mathcal U}(g))$, note that 
$k(f)=j_E(g)\rest[\kappa^{++}]^{<\om}$,
so $k(f)(a)=j_E(g)(a)=x$, as desired.
\end{proof}

\begin{claim}\label{claim:k_1[K]_gens_filter} $k_1[K]$ generates a filter $K'$ 
which is $\mathbb Q^{M_1}=\mathbb P(\kappa+1,j_E(\kappa))^{M_1}$-generic over 
$M_1$.\end{claim}
\begin{proof}
Let $D\in M_1$ be an open dense subset of $\mathbb Q^{M_1}$.
Using the previous claim, fix $f:[(\kappa^{++})^{N_1}]^{<\om}\to N_1$
with $f\in N_1$, and fix $a$, with $D=k(f)(a)$. We may assume that $f(b)$ is an 
open dense
subset of $\mathbb Q^{N_1}$ for each $b\in[(\kappa^{++})^{N_1}]^{<\om}$.
By the closure of $\mathbb Q^{N_1}$ in $N_1$,
it follows that $E=\bigcap\mathrm{range}(f)$ is an open dense subset of 
$\mathbb Q^{N_1}$,
and therefore $E\inter K\neq\emptyset$. But clearly
$k(E)\sub D$, so we are done.
\end{proof}
Now let $G \times H$ be $\mathbb P(0,\kappa) \times 
\add(\kappa^+,\kappa^{++})$-generic over $V$.  With this, we can 
simultaneously lift the embeddings $j_E$ and $j_{\mathcal U}$ to the domain 
$V[G][H]$.  Let $h = H \restriction (\kappa^{++})^{N_1}$. Calling the lifted 
embeddings by the same names, let $j_{\mathcal U}(G) = G \times h \times K$, and 
let $j_E(G) = G \times H \times K'$.  Since $j_{\mathcal U}[H]$ and $j_{E}[H]$ 
generate generics $H',H''$ over $N_1,M_1$ respectively, the embeddings can be 
lifted further to $V[G][H]$ as desired. Moreover, we can
lift $k$ to the domain $N_1[G][h][K][H']$ by setting
\[ k(G\times h\times 
K\times H')=G\times H\times K'\times H''.\]
This works since the original maps 
commute,
and note commutativity is preserved, i.e. we still have $k\circ 
j_{\mathcal U}=j_E$.

These lifted elementary embeddings are generated by an ultrafilter $\mathcal 
U_0$ and an extender $E_0$ in $V[G][H]$.  We can compute $\mathcal U_0$ in 
$M_1[G,H,K'][H'']$ because it is determined by the objects 
$V_{\kappa+1},\mathcal U,G,h,K$. 

Now this is only one way of lifting the 
embeddings.  For each $\alpha \in [(\kappa^{++})^{N_1},\kappa^{++})$, there is 
an automorphism of $\add(\kappa^+,\kappa^{++})$ which has the effect of 
switching the $0^{th}$ and $\alpha^{th}$ generic subset of $\kappa^+$.  Let us 
call the resulting filter $H_\alpha$ and let $h_\alpha$ be its restriction to 
$(\kappa^{++})^{N_1}$.  The filters $H_\alpha$ have the same information, but 
the filters $h_\alpha$ know things that their peers don't.  We can modify the 
liftings by instead putting $j_{\mathcal U}(G) = G \times h_\alpha \times K$ and 
$j_{E}(G) = G \times H_\alpha \times K'$.  These correspond to ultrafilters and 
extenders $\mathcal U_\alpha$ and $E_\alpha$, whose associated embeddings factor 
through by maps $k_\alpha$.  We have:
\begin{align*}
j_{\mathcal U_\alpha} &:  V[G][H] \to N_1[G,h_\alpha,K][H'] \\
j_{E_\alpha} &:  V[G][H] \to M_1[G,H_\alpha,K'][H''] = M_1[G,H,K'][H'']  \\
k_\alpha &: N_1[G,h_\alpha,K][H'] \to M_1[G,H_\alpha,K'][H'']
\end{align*}

Let $\left<M_n^*\right>_{n\leq\om}$ be the iterates
of $V[G][H]$ generated by $E_0$. In particular,
$M_1^*=M_1[G,H,K'][H'']$. Then $M^* = M[G^*][H^*]$, where
\[ G^* \times H^* \subseteq \mathbb 
P(0,j_{0,\omega}(\kappa))^M \times 
\add(j_{0,\omega}(\kappa)^+,j_{0,\omega}(\kappa)^{++})^M \]
is generic over $M$.  
\begin{claim}
 $j_{E_\alpha}(E_0)=j_{E_0}(E_0)$. 
\end{claim}
\begin{proof}
We have $j_{E_\alpha}\rest V=j_E=j_{E_0}\rest V$,
and $E_a\sub (E_0)_a,(E_\alpha)_a$ for each $a \in [\kappa^{++}]^{<\omega}$, where $X_a$ denotes the ultrafilter of the extender $X$ indexed by $a$. Therefore $j_{E_\alpha}(E)=j_{E_0}(E)$, and $j_{E_0}(E)_a \subseteq  j_{E_0}(E_0)_a$ for each $a \in [j_E(\kappa^{++})]^{<\omega}$.
 So let
$j^*_{12}:M^*_1\to M^*_2$
be the ultrapower map (associated to the ultrapower
$\ult(M^*_1,j_{E_0}(E_0))=M^*_2$),
and $j^\alpha_{12}$  the ultrapower map 
associated to the ultrapower
$\ult(M_1^*,j_{E_\alpha}(E_0))$.
It suffices to see that
\[ j^\alpha_{12}(j_{E_0}(G\times H))=j^*_{12}(j_{E_0}(G\times H)); \]
that is, that
\[ j^\alpha_{12}(G\times H\times K'\times H'')=j^*_{12}(G\times H\times 
K'\times H'').\]
But
\begin{equation}\label{eqn:first_step} j_{E_0}(G\times H)=G\times H\times 
K'\times H'', \end{equation}
and applying $j_{E_0}$ to this equation gives
\[ j^*_{12}(G\times H\times K'\times H'')=j_{E_0}(G\times H\times K'\times 
H'')\]
\[ =G\times H\times K'\times H''\times j_{E_0}(K')\times j_{E_0}(H'').\]
\[ =G\times H\times K'\times H''\times j_{E_\alpha}(K')\times 
j_{E_\alpha}(H''),\]
because $K'\in V$ and $H''$ is the $M_1$-generic
generated by $j_E[H]$, but $j_{E_0}(H)=j_{E_\alpha}(H)=H''$.

Now applying $j_{E_\alpha}$ to equation (\ref{eqn:first_step})
gives
\[ j^\alpha_{12}(G\times H_\alpha\times K'\times H'')=j_{E_\alpha}(G\times 
H\times K'\times H'') \]
\[ =G\times H_\alpha\times K'\times H''\times j_{E_\alpha}(K')\times 
j_{E_\alpha}(H'').\]
But $\crit(j^\alpha_{12})=j_E(\kappa)>\kappa^{++}$, and therefore
we can replace the two $H_\alpha$'s in this equation with $H$'s,
and combined with the earlier equations, this yields
\[ j^\alpha_{12}(G\times H\times K'\times H'')=G\times H\times K'\times 
H''\times j_{E_\alpha}(K')\times j_{E_\alpha}(H'') \]
\[ =j^*_{12}(G\times H\times K'\times H''),\]
as desired.
\end{proof}

By the claim, the
$\omega^{th}$ iterate of $M_1^*$ by $j_{E_\alpha}(E_0)$
is again
$M^*$.  Therefore $j_{E_\alpha}$ yields a self-embedding of $M^*$, and 
$j_{E_\alpha}(G^*)$ is the same as $G^*$, except with the $0^{th}$ and 
$\alpha^{th}$ generic subsets of $\kappa^+$ switched.

Each $\mathcal U_\alpha$ is in $M^*$, since it is in $M_1[G,H,K'][H'']$ and of 
rank below the critical point of the next embedding.  We can build the internal 
ultrapower maps
\[ i^*_\alpha : M^* \to N^*_\alpha = \ult(M^*,\mathcal U_\alpha),\] 
and externally define a factor map $k^*_\alpha : N^*_\alpha \to M^*$ by 
$k^*_\alpha([f]_{\mathcal U_\alpha}) = j_{E_\alpha}(f)(\kappa)$.  Note that 
$i^*_\alpha \restriction M$ is the ultrapower map $i : M \to N$, and $k^*_\alpha 
\restriction N = k$.

For $(\kappa^{++})^N \leq \alpha < \kappa^{++}$, the slice of $i^*_\alpha(G^*)$ at coordinate $\kappa$ is $h_\alpha$.  Thus for $\alpha \neq \beta$, we have $N^*_\alpha \nsubseteq N^*_\beta \nsubseteq N^*_\alpha$.  Finally, we define $\pi_{\alpha,\beta} : N_\alpha \to N_\beta$ by $i_\beta^* \circ k_\alpha^*$.
\end{proof}

In all of the constructions considered so far, the mutually embeddable models are sets or classes of some background universe that are not countably closed.  However, this is not a necessary feature:

\begin{proposition}
Suppose $\kappa$ is +2-strong. Then there are two distinct $\kappa$-closed mutually embeddable transitive models of $\ZFC$. 
\end{proposition}
\begin{proof} 
Let us prove the claim first under the stronger hypothesis, that $\kappa$ is $2^\kappa$-supercompact.

Let $\mathcal W$ be a normal $\kappa$-complete ultrafilter over $\p_\kappa(2^\kappa)$, and let $\mathcal U$ be its derived ultrafilter over $\kappa$.  Let $\la M_i : i \leq \omega \ra$ be the iterates of $V$ by $\mathcal W$, with associated embeddings $j_{\alpha,\beta} : M_\alpha \to M_\beta$, and $M_\omega$ being the direct limit.  As before, $j_{\mathcal W} = j_{0,1}$ yields a self-embedding of $M_\omega$.

Now consider the sequence $\la j_{n,n+1}[j_{0,n}(2^\kappa)] : n < \omega \ra$.  It is shown in \cite{2019arXiv190506062E} that this sequence is generic over $M_\omega$ for the Prikry forcing associated to $j_{0,\omega}(\mathcal W)$, and that if $G$ is the corresponding generic filter, then $M_\omega[G]$ is $2^\kappa$-closed, and $j_{\mathcal W}$ lifts to a self-embedding of $M_\omega[G]$.  The forcing adds no bounded subsets of $j_{0,\omega}(\kappa)$.

We also have that $\mathcal U$ is in $M_1$, and thus is in $M_\omega$ since $\crit(j_{1,\omega}) > (2^\kappa)^+$.  Let $i : M_\omega[G] \to N$ be the ultrapower map by $\mathcal U$ as computed in $M_\omega[G]$.  Since $\mathcal U = \{ X \subseteq \kappa : \kappa \in j_{\mathcal W}(X) \}$, there is as before a factor map $k : N \to M_\omega[G]$ such that $j_{\mathcal W} \restriction M_\omega[G] = k \circ i$.  Since $M_\omega[G]$ believes that $N$ is $\kappa$-closed, and $M_\omega[G]$ is itself $2^\kappa$-closed, both models are $\kappa$-closed.

Let us now reduce the large cardinal strength. 
Let $\kappa$ be $+2$-strong. Let $E$ be a $(\kappa,\lambda)$-extender witnessing it and let $\mathcal{U}$ be the derived ultrafilter from $E$. Let $a \in [\lambda]^{<\omega}$ and $f \colon \kappa \to V_\kappa$, be such that $[f, a]_E = \mathcal U$. Without loss of generality, $\kappa \in a$. Let us look at $E_a$, which is a measure on $\kappa^{|a|}$. Then, we can conclude that $\mathcal{U} = [f]_{E_{a}}$. 

Let $\mathbb{P}_{E_a}$ be the tree Prikry forcing for the $\kappa$-complete ultrafilter $E  _a$. Let us iterate $V$ using the measure $E_a$.  
By applying the same arguments as before, the sequence $P_a = \langle a, j_1(a), j_2(a), \dots\rangle$ is $M_{\omega}$-generic Prikry sequence for the forcing $j_\omega(\mathbb{P}_{E_a})$, and $M_{\omega}[P_a]$ is closed under $\kappa$-sequences. Moreover, $j_{E_a}(P_a) = \langle j_{n + 1}(a) \mid n < \omega\rangle$ and thus, $M_{\omega}[P_a] = M_{\omega}[j_E(P_a)]$ and thus we conclude as before that $j_{E} \restriction M_{\omega}[P_a]$ is a self map of $M_{\omega}[P_a]$, which splits through the measure $\mathcal U$. 
\end{proof} 


\section{Open questions}

\begin{question}
What is the exact consistency strength of the statement that there exists an inner model $M$ and an elementary embedding $j : M \to M$ such that the derived ultrafilter $\mathcal U = \{ X \subseteq \crit(j) : \crit(j) \in j(X) \} \in M$?
\end{question}

\begin{question}
What is the consistency strength of the statement that there are two distinct countably closed mutually embeddable inner models?
\end{question}

\begin{question}
Is it consistent to have two distinct countably closed inner models $M,N$ and embeddings $i : M \to N$, $j : N \to M$ such that $i \rest\Ord = j\rest\Ord$?
\end{question}

\bibliographystyle{amsplain.bst}
\bibliography{mutualembed}

\end{document}